\documentclass[12pt]{amsart}
\usepackage[initials]{amsrefs}
\usepackage{amssymb,graphicx,amscd,color,amsmath,amsfonts,amssymb,latexsym,amsthm,units,paralist}
\usepackage{color,geometry}
\usepackage[all]{xy}

\newtheorem{theorem}{Theorem}[section]

\newtheorem{remark}[theorem]{Remark}
\newtheorem{lemma}[theorem]{Lemma}

\newtheorem{definition}[theorem]{Definition}
\numberwithin{equation}{section}
\DeclareMathOperator{\ext}{ext\,\!}
\DeclareMathOperator{\sign}{sign\,\!}
\newcommand{\D}{\mathbb{D}}
\input epsf

\begin{document}

\title[On the constants in the polynomial BH inequality]{Sharp values for the constants in the polynomial Bohnenblust-Hille inequality}

\author[Jim\'{e}nez]{P. Jim\'{e}nez-Rodr\'{\i}guez}
\address{Departamento de An\'{a}lisis Matem\'{a}tico,\newline\indent Facultad de Ciencias Matem\'{a}ticas, \newline\indent Plaza de Ciencias 3, \newline\indent Universidad Complutense de Madrid,\newline\indent Madrid, 28040, Spain.}
\email{pablo$\_$jimenez@mat.ucm.es}

\author[Mu\~{n}oz]{G.A. Mu\~{n}oz-Fern\'{a}ndez}
\address{Departamento de An\'{a}lisis Matem\'{a}tico,\newline\indent Facultad de Ciencias Matem\'{a}ticas, \newline\indent Plaza de Ciencias 3, \newline\indent Universidad Complutense de Madrid,\newline\indent Madrid, 28040, Spain.}
\email{gustavo$\_$fernandez@mat.ucm.es}

\author[Murillo]{M. Murillo-Arcila}
\address{Instituto Universitario de Matem\'{a}tica Pura y Aplicada,\newline\indent Universitat Polit\`{e}cnica de Val\`{e}ncia, \newline\indent 46022, Val\`{e}ncia, Spain.}
\email{mamuar1@posgrado.upv.es}

\author[Seoane]{J.B. Seoane-Sep\'{u}lveda}
\address{Departamento de An\'{a}lisis Matem\'{a}tico,\newline\indent Facultad de Ciencias Matem\'{a}ticas, \newline\indent Plaza de Ciencias 3, \newline\indent Universidad Complutense de Madrid,\newline\indent Madrid, 28040, Spain.}
\email{jseoane@mat.ucm.es}

\thanks{G.A. Mu\~{n}oz-Fern\'{a}ndez and J.B. Seoane-Sep\'{u}lveda were supported by MTM2012-34341. M. Murillo-Arcila was supported by a grant of the FPU program of MEC, and by MTM2013-47093.}



\maketitle

\begin{abstract}
In this paper we prove that the complex polynomial Bohnenblust-Hille constant for $2$-homogeneous polynomials in ${\mathbb C}^2$ is exactly $\sqrt[4]{\frac{3}{2}}$. We also give the exact value of the real polynomial Bohnenblust-Hille constant for $2$-homogeneous polynomials in ${\mathbb R}^2$. Finally, we provide lower estimates for the real polynomial Bohnenblust-Hille constant for polynomials in ${\mathbb R}^2$ of higher degrees.
\end{abstract}


\section{Preliminaries: What you need to know}
Any homogeneous polynomial in ${\mathbb{K}}^{n}$ (${\mathbb{K}}={\mathbb{R}}$ or ${\mathbb{C}}$) of degree $m\in{\mathbb N}$ can be written as
    \begin{equation}\label{equ:hom_pol}
    P(x)={\sum\limits_{\left\vert \alpha\right\vert =m}}a_{\alpha}x^{\alpha},
    \end{equation}
where $x=(x_{1},\ldots,x_{n})\in{\mathbb{K}}^{n}$, $\alpha=(\alpha_{1}%
,\ldots,\alpha_{n})\in({\mathbb{N}}\cup\{0\})^{n}$, $|\alpha|=\alpha
_{1}+\cdots+\alpha_{n}$, $x^{\alpha}=x_{1}^{\alpha_{1}}\cdots x_{n}%
^{\alpha_{n}}$ and $a_{\alpha}\in{\mathbb K}$. Here, ${\mathcal P}(^m{\mathbb K}^n)$ stands for the finite dimension linear space of all the homogeneous polynomials of degree $m$ on ${\mathbb K}^n$.

If $\|\cdot\|$ is a norm on ${\mathbb K}^n$,
then the formula
    $$
    \|P\|:=\sup\{|P(x)|:x\in{\mathsf B}_X\},
    $$
for all $P\in {\mathcal P}(^m{\mathbb K}^n)$, where ${\mathsf B}_X$ is the unit ball of the Banach space $X=({\mathbb K}^n,\|\cdot\|)$, defines a norm in ${\mathcal P}(^m{\mathbb K}^n)$ usually called polynomial norm. The space ${\mathcal P}(^m{\mathbb K}^n)$ endowed with the polynomial norm induced by $X$ is denoted by ${\mathcal P}(^mX)$.

Other norms customarily used in ${\mathcal P}(^m{\mathbb K}^n)$ besides the polynomial norm are the $\ell_{p}$ norms of the coefficients. Namely, if $P$ is as in \eqref{equ:hom_pol} and $p\geq 1$, then
    $$
    |P|_p:=\left(\sum_{|\alpha|=m}|a_\alpha|^p\right)^\frac{1}{p},
    $$
defines another norm in ${\mathcal P}(^m{\mathbb K}^n)$. It is important to point out that although the polynomial norm is, most of the times, very difficult to compute, the $\ell_p$ norm of the coefficients is fairly easy to obtain. Since ${\mathcal P}(^m{\mathbb K}^n)$ is finite dimensional,
the polynomial norm $\|\cdot\|$ and the
$\ell_{p}$ norm $|\cdot|_{p}$ ($p\geq 1$) are equivalent, and therefore there
exist constants $k(m,n),\ K(m,n)>0$ such that
\begin{equation}
\label{equ:equiv}k(m,n)|P|_{p}\leq\|P\|\leq K(m,n)|P|_{p},
\end{equation}
for all $P\in{\mathcal{P}}(^{m}{\mathbb K}^n)$. The latter inequalities may provide a good
estimate on $\|P\|$ as long as we know the exact value of the best possible
constants $k(m,n)$ and $K(m,n)$ appearing in \eqref{equ:equiv}.

The problem presented above is an extension of the the well known
polynomial Bohnen\-blust-Hille inequality (polynomial BH inequality for short). It was proved in \cite{bh} that there
exists a constant $D_{m}>0$ such that for every $P\in{\mathcal{P}}(^{m}%
\ell_{\infty}^{n}({\mathbb K}))$ we have
\begin{equation}
|P|_{\frac{2m}{m+1}}\leq D_{m}\Vert P\Vert,\label{equ:BH}
\end{equation}
where $\ell_{\infty}^{n}({\mathbb R})$ and $\ell_{\infty}^{n}({\mathbb C})$ are, respectively, the real and complex versions of $\ell_{\infty}^{n}$.
Observe that \eqref{equ:BH} coincides with the first inequality in
\eqref{equ:equiv} for $p=\frac{2m}{m+1}$ except for the fact that $D_{m}$ in
\eqref{equ:BH} can be chosen in such a way that it is independent from the
dimension $n$. Actually Bohnenblust and Hille showed that $\frac{2m}{m+1}$ is
optimal in \eqref{equ:BH} in the sense that for $p<\frac{2m}{m+1}$, any
constant $D$ fitting in the inequality
\[
|P|_{p}\leq D\Vert P\Vert,
\]
for all $P\in{\mathcal{P}}(^{m}\ell_{\infty}^{n}({\mathbb K}))$ depends necessarily on $n$.

The best constants in \eqref{equ:BH} depend considerably on whether we consider the real or the complex version of $\ell_{\infty}^{n}$, which motivates the following definition:

\begin{definition}
The polynomial Bohnenblus-Hille constant for polynomials of degree $m$ is defined as
\[
D_{{\mathbb{K}},m}:=\inf\left\{  D>0:\text{$|P|_{\frac{2m}{m+1}}\leq
D\Vert P\Vert$, for all $n\in{\mathbb N}$ and $P\in{\mathcal{P}}(^{m}\ell_{\infty}^{n}({\mathbb K}))$}\right\}  .
\]
If we restrict attention to a certain subset $E$ of ${\mathcal P}(^m\ell_\infty^n({\mathbb K}))$ for some $n\in{\mathbb N}$, then we define
    $$
    D_{{\mathbb{K}},m}(E):=\inf\left\{  D>0:\text{$|P|_{\frac{2m}{m+1}}\leq
    D\Vert P\Vert$ for all $P\in E$}\right\}  .
    $$
For simplicity we will often use the notation $D_{{\mathbb{K}},m}(n)$ instead of $D_{{\mathbb{K}},m}({\mathcal{P}}(^{m}\ell_{\infty}^{n}({\mathbb K})))$. Note that
    $$
    1\leq D_{{\mathbb{K}},m}(n)\leq D_{{\mathbb{K}},m},
    $$
for all $m,n\in{\mathbb N}$.
\end{definition}

A good idea of the asymptotic growth of the constants $D_{{\mathbb{K}},m}$ and $D_{{\mathbb{K}},m}(n)$ is provided by the following definition:

\begin{definition}
The asymptotic hypercontractivity constant of the
polynomial BH inequality is
    $$
    H_{\mathbb{K},\infty}:=\limsup_{m} \sqrt[m]{D_{{\mathbb{K}},m}}.
    $$
Similarly, if we restrict attention to polynomials in $n$ variables then we define
    $$
    H_{\mathbb{K},\infty}(n):=\limsup_{m} \sqrt[m]{D_{{\mathbb{K}},m}(n)}.
    $$
Of course $1\leq H_{\mathbb{K},\infty}(n)\leq H_{\mathbb{K},\infty}$, for all $n\in{\mathbb N}$.
\end{definition}

It was shown in \cite{annals2011} that the complex polynomial
Bohnenblust--Hille inequality is, at most, hypercontractive. In \cite{bayart} the estimate on $D_{{\mathbb{C}},m}$ was improved. In fact the authors show that for every $B>1$ there exists $A>0$ such that $D_{{\mathbb{C}},m}\leq Ae^{B\sqrt{m\log m}}$, from which it follows that $H_{\mathbb{C},\infty}(n)=H_{\mathbb{C},\infty}=1$, for all $n\in{\mathbb N}$. For the real case, it has been recently proved in \cite{LAA2015} that $H_{\mathbb{R},\infty}=2$. However, not many exact values of $D_{{\mathbb{K}},m}(n)$ are known so far. This paper is devoted to calculate, explicitly or numerically some values of these constants.

This paper is arranged in two main sections. In Section \ref{sec2}, we employ some results on the geometry of spaces of polynomials in order to provide the exact value of $D_{\mathbb{C},2}(2)$. In Section \ref{sec3} we use a similar technique to find the exact value of $D_{\mathbb{R},2}(2)$. We also provide lower estimates for $D_{{\mathbb R},m}(2)$ and $H_{{\mathbb R},\infty}(2)$ by means of numerical calculus.

The polynomial Bohnenblust--Hille inequality has important applications in different fields of Mathematics and Physics and has been studied in depth by many authors since a multilinear version of the Bohnenblust--Hille inequality was proved in 1931 (see \cite{bayart,Blas, bom, dgm, ff, dmp, defant55, bh,Boas,defant2,DD,fr,monta,Nun2,seip3, Nu3,diniz2,lama11,addqq,psseo}) and the references therein.

\section{The exact value of $D_{\mathbb{C},2}(2)$}\label{sec2}

Throughout this section we will often identify any two-variable polynomial $az^2+bwz+cw^2$ or any one-variable polynomial $a\lambda^2+b\lambda +c$, for $a,b,c\in{\mathbb K}$, with the vector $(a,b,c)\in{\mathbb K}^3$. Also, we use the standard notation $\|az^2+bwz+cw^2\|_{\D}$ for the supremum of $|az^2+bwz+cw^2|$ for $z,w$ in the unit disk ${\mathbb D}$ of ${\mathbb C}$. Similarly, $\|a\lambda^2+b\lambda +c\|_{\mathbb D}$ stands for the supremum of $|a\lambda^2+b\lambda +c|$ for $\lambda\in{\mathbb D}$. Observe that
    $$
    \|az^2+bwz+cw^2\|_{\D} =\|a\lambda^2+b\lambda+c\|_{\mathbb D}=\max_{|\lambda|=1}|a\lambda^2+b\lambda+c|,
    $$
being the last of the latter equalities due to the Maximum Modulus Principle.

The main result of this section depends upon the following lemma, which is of independent interest.

\begin{lemma}\label{lem:Pablito}
Let $a,b,c \in \mathbb{C}$. There exist $a', b', c'\in \mathbb{R}$ such that $$\|az^2+bwz+cw^2\|_{\D}\geq\|a' z^2+b' zw+c'w^2\|_{\D} \quad \mbox{ and } \quad \|(a,b,c)\|_{4 \over 3}=\|(a',b',c')\|_{4 \over 3}.$$
\end{lemma}

\begin{proof}
If we perform the change of variables 
    $$
    z\mapsto ze^{-\frac{i\arg(a)}{2}}\quad \text{and}\quad w\mapsto we^{-\frac{i\arg(c)}{2}},
    $$
in $\|az^2+bzw+cw^2\|_{\D}$, we can assume (without loss of generality) that $a,c\geq 0$. We can also assume that $a\geq c$ by swapping $z$ and $w$. We have:
$$
\begin{aligned}
|a\lambda^2+b\lambda+c|^2&=\left( a\lambda^2+b\lambda+c\right)\left( a\overline{\lambda^2}+\overline{b\lambda}+c\right) \\
&=a^2+ac\lambda^2+a\overline{b}\lambda+ac\overline{\lambda^2}+c^2+c\overline{b\lambda}+ab\overline{\lambda}+bc\lambda+|b|^2 \\
&=a^2+c^2+|b|^2+2\left[ ac\text{Re}(\lambda^2)+a\text{Re}(\overline{b}\lambda)+c\text{Re}(b\lambda)\right] \\
&=a^2+c^2+|b|^2+2\left[ ac\text{Re}(\lambda^2)+(a+c)\text{Re}(b)\text{Re}(\lambda)+(a-c)\text{Im}(b)\text{Im}(\lambda)\right].
\end{aligned}
$$
Similarly, if $a', b', c'$ are real numbers, then:
$$
|a'\lambda^2+b'\lambda+c'|^2=a'^2+c'^2+|b'|^2+2\left[ a'c'\text{Re}(\lambda^2)+(a'+c')b'\text{Re}(\lambda)\right].
$$
\begin{enumerate}

\item Assume first $a\geq c \geq |b|$. Then, choose
$$
a'={(c^{4 \over 3}+|b|^{4 \over 3})^{3 \over 4} \over 2^{3 \over 4}}, \, c'=-a', \, b'=a.
$$
Then, $\|(a,b,c)\|_{4 \over 3}=\|(a',b',c')\|_{4 \over 3}$. On the other hand,
$$
a'^2+c'^2+|b'|^2+2\left[ a'c'\text{Re}(\lambda^2)+(a'+c')\text{Re}(\lambda)\right]=2a'^2+a^2-2a'^2 \text{Re}(\lambda^2),
$$
so that it is easy to see
$$
\|(a',b',c')\|^2_{\D}=4a'^2+a^2=\sqrt{2}(c^{4 \over 3}+|b|^{4 \over 3})^{3 \over 2}+a^2.
$$
Also, giving the value $\lambda=1$ (if $\text{Re}(b) \geq 0$) or $\lambda=-1$ (if Re$(b) \leq 0$), we can see that
$$
\|az^2+cw^2+bzw\|_{\D}^2 \geq a^2+c^2+|b|^2+2ac.
$$
Now, we want
$$
\sqrt{2}(c^{4 \over 3}+|b|^{4 \over 3})^{3 \over 2}+a^2 \leq a^2+c^2+|b|^2+2ac,
$$
that is,
$$
{c^2+|b|^2+2ac \over \sqrt{2}(c^{4 \over 3}+|b|^{4 \over 3})^{3 \over 2}} \geq 1.
$$
Divide both, numerator and denominator, by $a^2$ in order to convert the problem in having to achieve
$$
{x^2+y^2+2x \over \sqrt{2}(x^{4 \over 3}+y^{4 \over 3})^{3 \over 2}} \geq 1,
$$
for $0 \leq y \leq x \leq 1$.

\item Assume next $a \geq |b| \geq c$. In this seciond part of the proof we shall need to employ a couple of real valued functions that will come in handy to achieve our purpose. Let us first focus our attention on the choice of the constants $a', \, b', \, c'$, as before,
$$
a'={\left(|c|^{4 \over 3}+|b|^{4 \over 3}\right)^{3 \over 4} \over 2^{3 \over 4}}, \, c'=-a', \, b'=a.
$$
In this case, choose
$$
\lambda=\text{sign(Re($b$))}\sqrt{1 \over 2}+i\text{sign(Im($b$))}\sqrt{1 \over 2}.
$$
Then,

$$
\begin{aligned}
\|(a,b,c)\|_{\D} & \geq a^2+c^2+|b|^2+2\left[\sqrt{1 \over 2}(a+c)|\text{Re}(b)| + \sqrt{1 \over 2}(a-c) |\text{Im}(b)|\right]\\
&\geq a^2+c^2+|b|^2+\sqrt{2}|b|(a-c).
\end{aligned}
$$
Hence, we will achieve the desired result if we can guarantee
$$
\sqrt{2}\left(|c|^{4 \over 3}+|b|^{4 \over 3}\right)^{3 \over 2}+a^2 \leq a^2+c^2+|b|^2+\sqrt{2}(a-c)|b|,
$$
in other words,
$$
1 \leq \Phi_1(x,y):={x^2+y^2+\sqrt{2}(1-x)y \over \sqrt{2}\left(x^{4 \over 3}+y^{4 \over 3}\right)^{3 \over 2}},
$$
where $0 \leq x \leq y \leq 1$. \\

Let us focus now in another choice of constants $a', \, b', \, c'$:
$$
a'={\left(|a|^{4 \over 3}+|c|^{4 \over 3}+|b|^{4 \over 3}\right)^{3 \over 4} \over (2+k^{4/3})^{3 \over 4}}, \, c'=-a', \, b'=ka,
$$
where $k$ has been chosen so that $\|(a,b,c)\|_{4 \over 3}=\|(a',b',c')\|_{4 \over 3}$. It can be proved that $k\approx 2.828$.
Still giving the value $\lambda=\text{sign(Re($b$))}$
$$
\|(a,b,c)\|_{\D} \geq a^2+c^2+|b|^2+2ac,
$$
and again we guarantee that we achieve what we are searching for if we get
$$
4a'^2+b'^2={\left(|a|^{4 \over 3}+|c|^{4 \over 3}+|b|^{4 \over 3}\right)^{3 \over 2} \over (2+k^{4/3})^{3 \over 2}}(4+k^2) \leq a^2+c^2+|b|^2+2ac,
$$
in other words,
$$
1 \leq \Psi_1(x,y):={\left(1+x^2+y^2+2yx\right)(2+k^{4/3})^{3 \over 2} \over \left(4+k^2 \right) \left(y^{4 \over 3}+x^{4 \over 3}+1\right)^{3 \over 2}},
$$
with $0 \leq x \leq y \leq 1$. \\
The reader can check using elementary calculus that, if $$H(x,y) := \max\{\Phi_1(x,y), \, \Psi_1(x,y) \},$$ then
$$
1 \leq H(x,y) \text{ for every } \quad 0 \leq x \leq y \leq 1.
$$

\item Assume finally $|b| \geq a \geq c$. Then, we may choose
$$
a'={\left(|a|^{4 \over 3}+|c|^{4 \over 3}+|b|^{4 \over 3}\right)^{3 \over 4} \over (2+k^{4/3})^{3 \over 4}}, \, c'=-a', \, b'=ka,
$$
where $k$ is chosen as in the previous case. For $\lambda=\text{sign(Re($b$))}$, we still need to make sure that
$$
1 \leq \Phi_2(x,y):={(1+x^2+y^2+2xy)(2+k^{4/3})^{3 \over 2} \over (4+k^2)(x^{4/3}+y^{4/3}+1)^{3 \over 2}}.
$$
For $\lambda=\text{sign(Re($b$))}\sqrt{1 \over 2}+i\text{sign(Im($b$))}\sqrt{1 \over 2}$, we need to make sure that
$$
1 \leq \Psi_2(x,y):={(1+x^2+y^2+\sqrt{2}(y-x)(2+k^{4/3})^{3 \over 2} \over (4+k^2)(x^{4/3}+y^{4/3}+1)^{3 \over 2}}.
$$
Next, choose
$$
a'={(|a|^{4/3}+|c|^{4/3})^{3/4} \over 2^{3/4}}, \quad c'=-a' \quad \text{and} \quad b'=|b|,
$$
such that
$$
\|(a',b',c')\|^2_{\D}=4{(|a|^{4/3}+|c|^{4/3})^{3/2} \over 2^{3/2}}+|b|^2,
$$
and
$$
\lambda=\text{sign(Re($b$))}\sqrt{1-\left({a-c \over |b|}\right)^2}+i\, \text{sign(Im($b$))}{a-c \over |b|}.
$$
In that case,
$$
\begin{aligned}
\|(a,b,c)\|_{\D} & \geq a^2+c^2+|b|^2+2\left[ac\text{Re($\lambda^2$)}+(a+c)\text{Re($b$)Re($\lambda$)}+(a-c)\text{Im($b$)Im($\lambda$)}\right]   \\
& =a^2+c^2+|b|^2+2\left[ac\left(1-2\left({a-c \over |b|}\right)^2\right)+(a+c)|\text{Re($b$)}|\sqrt{1-\left({a-c \over |b|}\right)^2}\right. \\
& \left. \quad +(a-c)|\text{Im($b$)}|{a-c \over |b|} \right] .
\end{aligned}
$$
Assume first $|\text{Im($b$)}|\geq {\sqrt{2} \over 2}$. Then,
$$
\begin{aligned}
\|(a,b,c)\|_{\D} & \geq a^2+c^2+|b|^2\\
& + 2\left[ac\left(1-2\left({a-c \over |b|}\right)^2\right)+(a-c){\sqrt{2} \over 2}|b|{a-c \over |b|}\right].
\end{aligned}
$$

Hence, we will achieve what we are searching for if we can assure that
$$
1 \leq \Omega_2^{(1)}(x,y):={x^2+y^2+2xy(1-2(y-x)^2)+\sqrt{2}(y-x)^2 \over \sqrt{2}(x^{4/3}+y^{4/3})^{3 \over 2}}.
$$
On in on, we need to prove
$$
1 \leq \max\{\Phi_2(x,y), \, \Psi_2(x,y), \, \Omega_2^{(1)}(x,y) \},
$$
for $0 \leq x \leq y \leq 1$. This can be done by means of elementary calculus, and we leave it as an exercise to the reader.\\

\noindent On the other hand if, instead, we have $|\text{Re($b$)}|\geq {\sqrt{2} \over 2}$, then

$$
\begin{aligned}
\|(a,b,c)\|_{\D} & \geq a^2+c^2+|b|^2\\ & +2\left[ac\left(1-2\left({a-c \over |b|}\right)^2\right)+(a+c){\sqrt{2} \over 2}|b|\sqrt{1-\left({a-c \over |b|}\right)^2}\right],
\end{aligned}
$$

and (in this case) we will be working with the condition
$$
1 \leq \Omega_2^{(2)}(x,y):={x^2+y^2+2xy(1-2(y-x)^2)+\sqrt{2}(y+x)\sqrt{1-(y-x)^2} \over \sqrt{2}(x^{4/3}+y^{4/3})^{3 \over 2}},
$$
and, in conclusion, we shall need to guarantee that
$$
1 \leq \max\{\Phi_2(x,y), \, \Psi_2(x,y), \, \Omega_2^{(2)}(x,y) \},
$$
for $0 \leq x \leq y \leq 1$, which we also leave as an exercise to the reader.
\end{enumerate}
And, with this last case, the proof is complete.
\end{proof}

In order to prove that $D_{{\mathbb C},2}(2)=\sqrt[4]{\frac{3}{2}}$ we will also need the following description of the extreme points of the unit ball of ${\mathbb R}^3$ endowed with the norm
$$
\|(a,b,c)\|_{\mathbb D}:=\sup\{|az^2+bz+c|:|z|\leq 1\}
$$
for $a,b,c\in{\mathbb R}$.
This norm has been studied by Aron and klimek in \cite{AronKlimek}, where they denote it by $\|\cdot\|_{\mathbb C}$. Observe, once again that
    $$
    \|(a,b,c)\|_{\mathbb D}=\|az^2+bwz+cz^2\|_{\D}.
    $$

\begin{theorem}[Aron and Klimek, \cite{AronKlimek}]
Let $E_{\mathbb R}$ be the real subspace of ${\mathcal P}(^2\ell_\infty^2({\mathbb C}))$ given by
$\{az^2+bwz+cw^2:\ (a,b,c)\in{\mathbb R}^3\}$. Then
    $$
    \ext({\mathsf B}_{E_{\mathbb R}})=\left\{\left(s,\sqrt{4|s||t|\left(\frac{1}{(|s|+|t|)^2}-1\right)},t\right):(s,t)\in G\right\},
    $$
where $\ext({\mathsf B}_{E_{\mathbb R}})$ is the set of extreme points of the unit ball of $E_{\mathbb R}$, namely ${\mathsf B}_{E_{\mathbb R}}$ and $G=\{(s,t)\in{\mathbb R}^2:\text{$|s|+|t|<1$ and $|s+t|\leq (s+t)^2$}\}\cup\{\pm(1,0),\pm(0,1)\}$.
\end{theorem}
\begin{theorem}
The optimal complex polynomial Bohnenblust-Hille constant for polynomials in $E_{\mathbb R}$, which we denote by $D_{{\mathbb C},2}(E_{\mathbb R})$, is given by $D_{{\mathbb C},2}(E_{\mathbb R})=\sqrt[4]{\frac{3}{2}}$. Moreover,
    $$
    D_{{\mathbb C},2}(2)= \sqrt[4]{\frac{3}{2}}\approx 1.1066.
    $$
\end{theorem}

\begin{proof}
Using convexity we have
    \begin{align*}
    D_{{\mathbb C},2}(E_{\mathbb R})&=\sup\{\|(a,b,c)\|_{\frac{4}{3}}:\|az_{1}^{2}+bz_{1}z_{2}+cz_{2}^{2}\|\leq 1\}\\
    &=\sup\{\|(a,b,c)\|_{\frac{4}{3}}:\|(a,b,c)\|_{\mathbb C}\leq 1\}\\
    &=\sup\{\|(a,b,c)\|_{\frac{4}{3}}:(a,b,c)\in \ext({\mathsf B}_{E_{\mathbb R}})\},
    \end{align*}
Hence
    $$
    D_{{\mathbb C},2}(E_{\mathbb R})=\sup\left\{\left(|s|^\frac{4}{3}+|t|^\frac{4}{3}+\left[4|s||t|\left(\frac{1}{(|s|+|t|)^2}-1\right)\right]^\frac{2}{3}\right)^\frac{3}{4}:(s,t)\in G\right\}.
    $$
If $\Phi(s,t)=\left(|s|^\frac{4}{3}+|t|^\frac{4}{3}+\left[4|s||t|\left(\frac{1}{(|s|+|t|)^2}-1\right)\right]^\frac{2}{3}\right)^\frac{3}{4}$ for $(s,t)\in G$, one can prove using elementary calculus that $\Phi$ attains its maximum on $G$ at $\pm\left(\frac{\sqrt{3}}{6},-\frac{\sqrt{3}}{6}\right)$ and $\Phi\left(\frac{\sqrt{3}}{6},-\frac{\sqrt{3}}{6}\right)=\sqrt[4]{\frac{3}{2}}$. Finally, from Lemma \ref{lem:Pablito} we also obtain that $D_{{\mathbb C},2}(2)= \sqrt[4]{\frac{3}{2}}\approx 1.1066$.
\end{proof}

The reader can find a sketch of the graph of $\Phi$ on the part of $G$ contained in the second quadrant in Figure \ref{fig:Phi}.
\begin{figure}
\centering
\includegraphics[height=.6\textwidth,keepaspectratio=true]{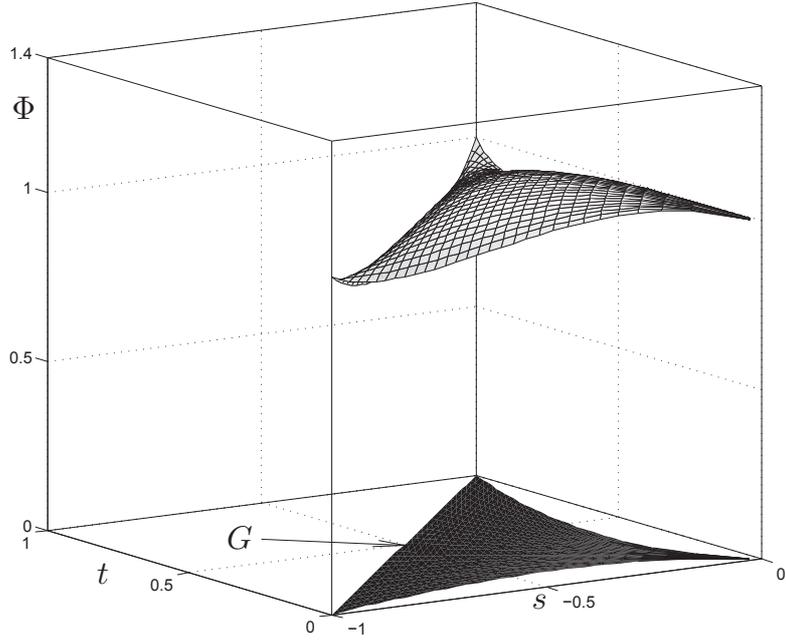}
\caption{Graph of $\Phi$ on $G$ intersected with the second quadrant.}\label{fig:Phi}
\end{figure}

\section{The exact value of $D_{\mathbb{R},2}(2)$ and lower bounds for $D_{{\mathbb R},m}(2)$}\label{sec3}

In \cite{LAA2015} it is proved that the asymptotic hypercontractivity constant of the real polynomial BH inequality is exactly $2$. Is it true that $H_{\mathbb{R},\infty}(2)=2$? The results presented here suggest that, perhaps $H_{\mathbb{R},\infty}(2)<2$. In this section, as we did in the previous one, we will also identify polynomials with the vector of its coefficients.

\begin{remark}\label{rem:norms}
Throughout this section we will compute several times norms of polynomials on the real line numerically. This is done by using Matlab. In particular, if $P(x)$ is a real polynomial on ${\mathbb R}$, we apply the predefined Matlab function \verb"roots.m" to $P'$ in order to obtain an approximation of all the critical points of $P$. If $x_1,\ldots,x_k$ are all the roots of $P'$ in $[-1,1]$, then we approach the norm of $P$ as
    $$
    \|P\|:=\max\{|P(x)|:x\in[-1,1]\}=\max\{|P(x_i)|,|P(\pm 1)|:i=1,\ldots,k\}.
    $$
Another Matlab predefined function, namely \verb"conv.m", is used in order to to multiply polynomials. This is done to obtain Figure \ref{fig:graphs}.
\end{remark}

\subsection{The exact calculation of $D_{{\mathbb R},2}(2)$}

The value of the constant $D_{{\mathbb{R}%
},2}(2)$ can be obtained using the geometry
of the unit ball of ${\mathcal{P}}(\ell_{\infty}^{2}({\mathbb R}))$ described in
\cite{Choi}. We state the result we need for completeness:

\begin{theorem}
\label{the:Choi}[Choi, Kim \cite{Choi}] \label{the:ExtPoints} The set $\ext({\mathsf B}_{{\mathcal{P}}(^{2}\ell_{\infty}^{2}({\mathbb R}))})$ of extreme
points of the unit ball of ${\mathcal{P}}(^{2}\ell_{\infty}^{2}({\mathbb R}))$ is given by
\[
\ext({\mathsf B}_{{\mathcal{P}}(^{2}\ell_{\infty}^{2}({\mathbb R}))})=\{\pm x^{2},\ \pm y^{2},\ \pm(tx^{2}-ty^{2}\pm2\sqrt{t(1-t)}xy):\ \text{$t\in[1/2,1]$}\}.
\]
\end{theorem}

As a consequence of the previous result, we obtain the following:

\begin{theorem}
Let $f$ be the real valued function given by $$f(t) = \left[  2t^{\frac{4}{3}}+(2\sqrt{t(1-t)})^{\frac{4}{3}}\right]^{\frac{3}{4}}.$$
We have that $D_{{\mathbb{R}},2}(2) = f(t_0) \approx 1.837373$, where 
    $$
    t_0 = \frac{1}{36} \left(2 \sqrt[3]{107+9 \sqrt{129}}+\sqrt[3]{856-72 \sqrt{129}}+16\right) \approx 0.867835.
    $$
The exact value of $f(t_0)$ is given by
\begin{center}
{\scriptsize $\left(\frac{\left(2 \sqrt[3]{107+9 \sqrt{129}}+\sqrt[3]{856-72 \sqrt{129}}+16\right)^{4/3}}{18\ 6^{2/3}}+\frac{1}{9 \left(-\frac{3}{-2 \sqrt[3]{107+9 \sqrt{129}}+\left(107+9 \sqrt{129}\right)^{2/3}-2 \sqrt[3]{107-9 \sqrt{129}}+\left(107-9 \sqrt{129}\right)^{2/3}-60}\right)^{2/3}}\right)^{3/4},$}	
\end{center}

Moreover, the following normalized polynomials are extreme for this problem:
\[
P_2(x,y)=\pm(t_0 x^{2}- t_0 y^{2}\pm 2\sqrt{t_0(1-t_0)}xy).
\]
\end{theorem}

\begin{proof}
Let
$$f(t)=\left[  2t^{\frac{4}{3}}+(2\sqrt{t(1-t)})^{\frac{4}{3}}\right]
^{\frac{3}{4}}. $$
We just have to notice that due to the convexity of the $\ell_{p}$-norms and
Theorem \ref{the:Choi} we have
\begin{align*}
D_{{\mathbb{R}},2}(2) &  =\sup
\{|\mathbf{a}|_{\frac{4}{3}}:\mathbf{a}\in{\mathsf{B}}_{{\mathcal{P}}(^{2}%
\ell_{\infty}^{2}{\mathbb R})}\}\\
&  =\sup\{|\mathbf{a}|_{\frac{4}{3}}:\mathbf{a}\in
\ext({\mathsf B}_{{\mathcal{P}}(^{2}\ell_{\infty}^{2}{\mathbb R})})\}=\sup_{t\in
\lbrack1/2,1]}f(t).
\end{align*}
Some calculations will show that the last supremum is attained at $t=t_0$, concluding the proof.
\end{proof}

Now, if  $\mathbf{a}_{n}$ is the vector of the
coefficients of $P_{2}^{n}$ for each $n\in{\mathbb{N}}$, then we know that
\begin{equation}
D_{{\mathbb{R}},2n}(2)\geq\frac{|\mathbf{a}_{n}|_{\frac{4n}{2n+1}}}{\Vert
P_{2}\Vert^{n}}.\label{equ:case_2}%
\end{equation}
Since $\|P_2\|= 1$, then \eqref{equ:case_2} with $n=300$ (see also Figure \ref{fig:graphs}) proves that
    $$
    D_{\mathbb{R},600}(2)\geq(1.36117)^{600},
    $$
providing numerical evidence showing that
    $$
    H_{{\mathbb R},\infty}(2)\geq 1.36117.
    $$
    
\subsection{Educated guess for the exact calculation of $D_{{\mathbb R},3}(2)$}



To the authors' knowledge the calculation of $\|P\|$ is, in general, far from
being easy. However there is a way to compute $\|P\|$ for specific cases. For
instance Grecu, Mu\~{n}oz and Seoane prove in \cite[Lemma 3.12]{GMS} the
following formula:

\begin{lemma}\label{lem:P_3}
\label{lem:sup_norm_ab} If for every $a,b\in{\mathbb{R}}$ we define
$P_{a,b}(x,y)=ax^{3}+bx^{2}y+bxy^{2}+ay^{3}$ then
\[
\Vert P_{a,b}\Vert=%
\begin{cases}
\left\vert a-\frac{b^{2}}{3a}+\frac{2b^{3}}{27a^{2}}+\frac{2a}{27}\left(
-\frac{3b}{a}+\frac{b^{2}}{a^{2}}\right)  ^{\frac{3}{2}}\right\vert  &
\text{if $a\neq0$ and $b_{1}<\frac{b}{a}<3-2\sqrt{3}$},\\
\left\vert 2a+2b\right\vert  & \text{otherwise},
\end{cases}
\]
where
\[
b_{1}=\frac{3}{7}\left(  3-\frac{2\sqrt[3]{9}}{\sqrt[3]{-12+7\sqrt{3}}%
}+2\sqrt[3]{-36+21\sqrt{3}}\right)  \thickapprox-1.6692.
\]

\end{lemma}

From Lemma \ref{lem:sup_norm_ab} we have the
following sharp polynomial Bohnenblust-Hille type constant:

\begin{theorem}
Let $P_{a,b}(x,y)=ax^{3}+bx^{2}y+bxy^{2}+ay^{3}$ for $a,b\in{\mathbb{R}}$ and
consider the subset of ${\mathcal{P}}(^{3}\ell_{\infty}^{2}({\mathbb R}))$ given by
$E=\{P_{a,b}:a,b\in{\mathbb{R}}\}$. Then%
\[
\frac{|(a,b,b,a)|_{\frac{3}{2}}}{\Vert P_{a,b}\Vert}=\left\{
\begin{array}
[c]{c}%
\frac{27a^{2}\left(  2|a|^{\frac{3}{2}}+2|b|^{\frac{3}{2}}\right)  ^{\frac
{2}{3}}}{\left\vert 27a^{3}-9ab^{2}+2b^{3}+2\sign(a)\left(  -3ab+b^{2}\right)
^{\frac{3}{2}}\right\vert },\text{ if $a\neq0$ and $b_{1}<\frac{b}{a}%
<3-2\sqrt{3}$},\\
\frac{\left(  2|a|^{\frac{3}{2}}+2|b|^{\frac{3}{2}}\right)  ^{\frac{2}{3}}%
}{2|a+b|},\text{ otherwise}%
\end{array}
\right.
\]
where $b_{1}$ was defined in Lemma \ref{lem:sup_norm_ab}. Moreover, the above
function attains its maximum when $\frac{b}{a}=b_{1}$, which implies that
\[
D_{{\mathbb{R}},3}(E)=\frac{\left(  2+2|b_{1}|^{\frac{3}{2}}\right)
^{\frac{2}{3}}}{2|1+b_{1}|}\thickapprox 2.5525
\]
\end{theorem}

\begin{figure}
\centering
\includegraphics[height=.6\textwidth,keepaspectratio=true]{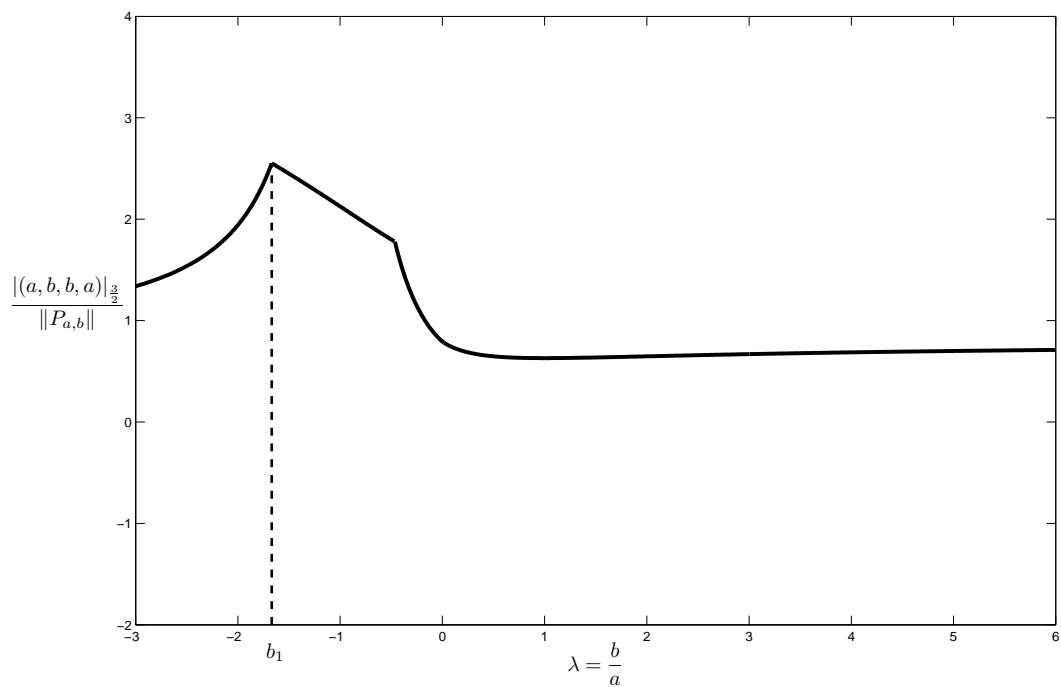}
\caption{Graph of the quotient $\frac{|(a,b,b,a)|_\frac{3}{2}}{\|P_{a,b}\|}$ as a function of $\lambda=\frac{a}{b}$.}\label{fig:Pab}
\end{figure}
The authors have numerical evidence to state that
\[
D_{{\mathbb{R}},3}(2)=D_{{\mathbb{R}}%
,3}(E).
\]
Moreover, one polynomial for which $D_{{\mathbb{R}},3}(2)$ would be attained is 
    $$
    P_{3}(x,y)=x^{3}+b_1x^{2}y+b_1xy^{2}+y^{3},
    $$
where $b_1\approx -1.6692$ is as in Lemma \ref{lem:P_3}.
It can be proved from Lemma \ref{lem:P_3} that
    $$
    \|P_3\|\approx 1.33848,
    $$
up to 5 decimal places. If $\mathbf{a_n}$ is the vector of the coefficients of $P_{3}(x,y)^{n}$ and we use the fact that
\begin{equation}
D_{{\mathbb{R}},3n}(2)\geq\frac{|\mathbf{a}_{n}|_{\frac{6n}{3n+1}}}{\Vert
P_{3}\Vert^{n}},\label{equ:case_3}%
\end{equation}
then putting $n=200$ in \eqref{equ:case_3} we obtain, for instance,
    $$
    D_{{\mathbb R},600}(2)\geq (1.42234)^{600},
    $$
which provides numerical evidence showing that
    $$
    H_{{\mathbb R},\infty}(2)\geq 1.42234.
    $$


\subsection{Numerical calculation of $D_{{\mathbb R},5}(2)$}

Let us define the polynomial
\[
P_{5}(x,y)=ax^{5}-bx^{4}y-cx^{3}y^{2}+cx^{2}y^{3}+bxy^{4}-ay^{5},
\]
with
\begin{align*}
a &  =0.19462,\\
b &  =0.66008,\\
c &  =0.97833.
\end{align*}
The norm of $P_5$ can be calculated numerically (using Remark \ref{rem:norms}), and it turns out to be
\[
\Vert P_{5}\Vert=0.28617,
\]
up to $5$ decimal places.
The authors have numerical evidence showing that
    $$
    D_{{\mathbb{R}}%
    ,5}(2)\approx 6.83591.
    $$
In any case we have
    $$
    D_{{\mathbb{R}},5}(2)\geq\frac{|(a,-b,-c,c,b,-a)|_{\frac{5}{3}}}{\Vert
    P_{5}\Vert}\approx 6.83591.
    $$
It is interesting to observe that we can improve numerically the estimate $H_{\infty,{\mathbb R}}(2)\geq \sqrt[8]{27}\approx 1.50980$ (see \cite[Theorem 4.2]{LAA2015}) by considering polynomials of the form $P_5^n$. Indeed, if  $\mathbf{a}_{n}$ is the vector of the
coefficients of $P_{5}^{n}$ for each $n\in{\mathbb{N}}$, then we know that
\begin{equation}
D_{{\mathbb{R}},5n}(2)\geq\frac{|\mathbf{a}_{n}|_{\frac{10n}{5n+1}}}{\Vert
P_{5}\Vert^{n}},\label{equ:case_5}%
\end{equation}
Using \eqref{equ:case_5} with $n=120$ we obtain, in particular (see also Figure \ref{fig:graphs})
    $$
    D_{\mathbb{R},600}(2)\geq(1.54987)^{600},
    $$
providing numerical evidence showing that
    $$
    H_{{\mathbb R},\infty}(2)\geq 1.54987.
    $$


\subsection{Educated guess for the exact calculation of $D_{{\mathbb R},6}(2)$}

The authors have numerical evidence pointing to the fact that an extreme
polynomial in the Bohnenblust-Hille inequality for polynomials in
${\mathcal{P}}(^{6}\ell_{\infty}^{2}({\mathbb R}))$ may be of the form
\[
Q_{a,b}(x,y)=ax^{5}y+bx^{3}y^{3}+axy^{5}.
\]
This motivates a deeper study of this type of polynomials, which we do in the
following result.

\begin{theorem}\label{the:case6}
Let $Q_{a,b}(x,y)=ax^5y+bx^3y^3+axy^5$ for $a,b\in{\mathbb{R}}$ and
consider the subspace of ${\mathcal{P}}(^{6}\ell_{\infty}^{2}({\mathbb R}))$ given by
$F=\{Q_{a,b}:a,b\in{\mathbb{R}}\}$. Suppose $\lambda_0<\lambda_1$ are the only two roots of the equation
    $$
    \frac{|3\lambda^2-20+\lambda\sqrt{9\lambda^2-20}|}{25}\sqrt{\frac{-3\lambda-\sqrt{9\lambda^2-20}}{10}}=|2+\lambda|.
    $$
Then if $\lambda =\frac{b}{a}$ we have

\[
\frac{|(0,a,0,b,0,a,0)|_{\frac{12}{7}}}{\Vert Q_{a,b}\Vert}=\left\{
\begin{array}
[c]{c}%
\frac{25\sqrt{10}\left(2+|\lambda|^{\frac{12}{7}}\right)^{\frac
{12}{7}}}{\left|3\lambda^2-20+\lambda\sqrt{9\lambda^2-20}\right|\sqrt{-3\lambda-\sqrt{9\lambda^2-20}} },\text{ if $a\neq0$ and $\lambda_0<\frac{b}{a}%
<\lambda_1$},\\
\frac{\left(  2+|\lambda|^{\frac{12}{7}}\right)^{\frac{7}{12}}
}{|2+\lambda|},\text{ otherwise.}%
\end{array}
\right.
\]
Observe that $\lambda_0\approx -2.2654$, $\lambda_1\approx -1.6779$ and the above
function attains its maximum when $\frac{b}{a}=\lambda_0$ (see Figure \ref{fig:Qab}), which implies that
\[
D_{{\mathbb{R}},6}(F)=\frac{\left(  2+|\lambda_{0}|^{\frac{12}{7}}\right)
^{\frac{12}{7}}}{|2+\lambda_{0}|}\thickapprox 10.7809.
\]
\end{theorem}

\begin{proof}
We do not lose generality by considering only polynomials of the form $Q_{1,\lambda}$, in which case
    $$
    \|Q_{1,\lambda}\|=\sup\{|x^5+\lambda x^3+x|:x\in[0,1]\}.
    $$
The polynomial $q_{\lambda}(x):=x^5+\lambda x^3+x$ has no critical points if $\lambda>-\frac{2\sqrt{5}}{3}$, otherwise it has the following critical points in $[0,1]$:
    $$
        x_0:=\sqrt{\frac{-3\lambda-\sqrt{9\lambda^2-20}}{10}}\quad\text{and}\quad
        x_1:=\sqrt{\frac{-3\lambda+\sqrt{9\lambda^2-20}}{10}}\quad\text{if $-2\leq \lambda\leq -\frac{2\sqrt{5}}{3}$,}
    $$
and $x_0$ if $\lambda\leq -2$. Notice that
    \begin{align*}
     q_\lambda(x_0)&= \frac{-3\lambda^2+20-\lambda\sqrt{9\lambda^2-20}}{20}x_0\\
     q_\lambda(x_1)&= \frac{-3\lambda^2+20+\lambda\sqrt{9\lambda^2-20}}{20}x_1.
    \end{align*}
It is easy to check that $|q_\lambda(x_0)|\geq |q_\lambda(x_1)|$ for $-2\leq \lambda\leq -\frac{2\sqrt{5}}{3}$, which implies that
    $$
    \|Q_{1,\lambda}\|=\begin{cases}
    \max\{|2+\lambda|,|q_\lambda(x_0)|\}&\text{if $-2\leq \lambda\leq -\frac{2\sqrt{5}}{3}$,}\\
    |2+\lambda|&\text{otherwise.}
    \end{cases}
    $$
The equation $|2+\lambda|=|q_\lambda(x_0)|$ turns out to have only two roots, namely $\lambda_0\approx -2.2654$ and $\lambda_1\approx -1.6779$. By continuity, it is easy to prove that $|2+\lambda|\leq |q_\lambda(x_0)|$ only if $-2\leq \lambda\leq -\frac{2\sqrt{5}}{3}$, which concludes the proof.
\end{proof}

As mentioned above, we have numerical evidence showing that
    $$
    D_{{\mathbb{R}},6}(2)=D_{{\mathbb{R}},6}(F)=\frac{\left(  2+|\lambda_{0}|^{\frac{12}{7}}\right)
^{\frac{12}{7}}}{|2+\lambda_{0}|}\thickapprox 10.7809.
    $$
In any case we do have that
\[
D_{{\mathbb{R}},6}(2)\geq 10.7809.
\]

\begin{figure}
\centering
\includegraphics[height=.6\textwidth,keepaspectratio=true]{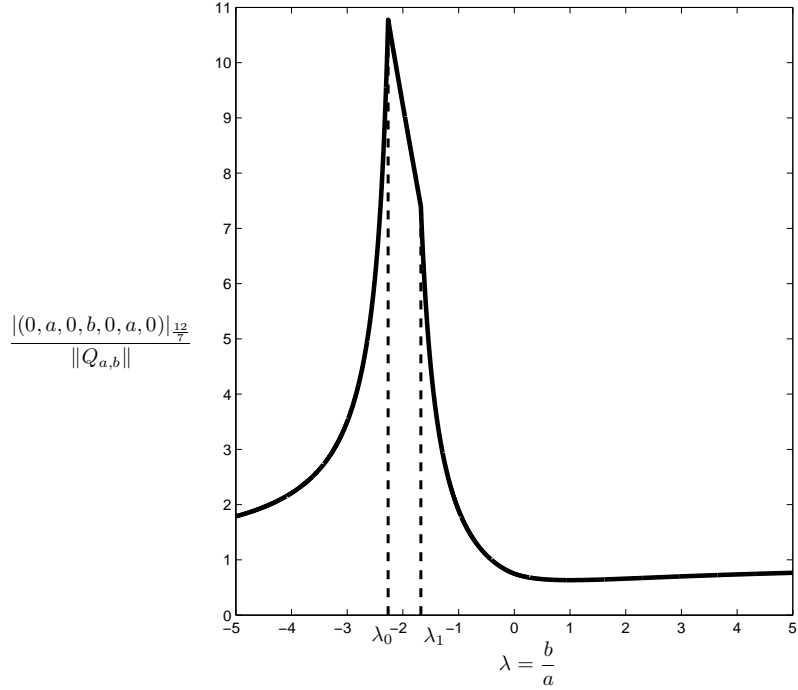}
\caption{Graph of the quotient $\frac{|(0,a,0,b,0,a,0)|_\frac{12}{7}}{\|Q_{a,b}\|}$ as a function of $\lambda=\frac{b}{a}$.}\label{fig:Qab}
\end{figure}

As we did in the previous cases, it would be interesting to know if we can
improve numerically our best lower bound on $H_{{\mathbb{R}},\infty}$ by considering powers of
\[
P_{6}(x,y)=Q_{1,\lambda_{0}}(x,y)=x^{5}y+\lambda_{0}x^{3}y^{3}+xy^{5},
\]
with $\lambda_0$ as in Theorem \ref{the:case6} ($\lambda_{0}\approx-2.2654$).
If $\mathbf{a}_{n}$ is the vector of the coefficients of $P_{6}^{n}$ for each
$n\in{\mathbb{N}}$, then we know that
\begin{equation}
\label{equ:case_6}D_{{\mathbb{R}},6n}(2)\geq\frac{|\mathbf{a}_{n}|_{\frac
{12n}{6n+1}}}{\|P_{6}\|^{n}}.
\end{equation}
Using \eqref{equ:case_6} with $n=100$ and estimating $\|P_6\|$ according to Remark \ref{rem:norms} we obtain
    \[
    D_{\mathbb{R},600}(2)\geq(1.58432)^{600},
    \]
which suggests that (see Figure \ref{fig:graphs})
    $$
    H_{\infty,\mathbb{R}}(2)\geq 1.58432.
    $$

\subsection{Numerical calculation of $D_{{\mathbb R},7}(2)$}
Let us define the polynomial
\[
P_{7}(x,y)=-ax^7+bx^6y+cx^5y^2-dx^4y^3-dx^3y^4+cx^2y^5+bxy^6-ay^7,
\]
with
\begin{align*}
a &  =0.05126,\\
b &  =0.22070,\\
c &  =0.50537,\\
d &  =0.71044.
\end{align*}
It can be proved numerically (using Remark \ref{rem:norms}) that
\[
\Vert P_{7}\Vert\approx 0.07138,
\]
up to $5$ decimal places. The authors have numerical evidence showing that
    $$
    D_{{\mathbb{R}},7}(2)\approx \frac{|(-a,b,c,-d,-d,c,b,-a)|_{\frac{7}{4}}}{\Vert
    P_{7}\Vert}\approx 19.96308.
    $$
If $\mathbf{a}_{n}$ is the vector of the
coefficients of $P_{7}^{n}$ for each $n\in{\mathbb{N}}$, then we know that
\begin{equation}
D_{{\mathbb{R}},7n}(2)\geq\frac{|\mathbf{a}_{n}|_{\frac{14n}{7n+1}}}{\Vert
P_{7}\Vert^{n}}.\label{equ:case_7}%
\end{equation}

Moreover, if we put $n=86$ in \eqref{equ:case_7} we obtain
    $$
    D_{\mathbb{R},602}(2)\geq (1.61725)^{602},
    $$
suggesting that
    $$
    H_{{\mathbb R},\infty}(2) \geq 1.61725.
    $$

\subsection{Numerical calculation of $D_{{\mathbb R},8}(2)$}
Let us define the polynomial
\[
P_{8}(x,y)=-ax^7y+bx^5y^3-bx^3y^5+axy^7,
\]
with
\begin{align*}
a &  =0.15258,\\
b &  =0.64697.
\end{align*}
It can be established numerically (see Remark \ref{rem:norms}) that \[
\| P_{8}\|\approx 0.02985,
\]
up to $5$ decimal places. 
The authors have numerical evidence showing  that
    $$
    D_{{\mathbb{R}},8}(2)\approx \frac{|(0,-a,0,b,0,-b,0,a,0)|_{\frac{16}{9}}}{\Vert
    P_{8}\Vert}\approx 33.36323.
    $$
If $\mathbf{a}_{n}$ is the vector of the
coefficients of $P_{8}^{n}$ for each $n\in{\mathbb{N}}$, then we know that
\begin{equation}
D_{{\mathbb{R}},8n}(2)\geq\frac{|\mathbf{a}_{n}|_{\frac{16n}{8n+1}}}{\Vert
P_{8}\Vert^{n}}.\label{equ:case_8}%
\end{equation}
Moreover, using \eqref{equ:case_8} with $n=75$ we obtain
    $$
    D_{\mathbb{R},600}(2)\geq (1.64042)^{600},
    $$
which suggests that
    $$
    H_{{\mathbb R},\infty}(2)\geq 1.64042.
    $$

\subsection{Numerical calculation of $D_{{\mathbb R},10}(2)$}
In this case our numerical estimates show that there exists an extreme polynomial in the Bohnenblust-Hille polynomial inequality in ${\mathcal P}(^{10}\ell_\infty^2({\mathbb R}))$ of the form
    $$
    P_{10}(x,y)=ax^9y+bx^7y^3+x^5y^5+bx^3y^7+axy^9,
    $$
with
\begin{align*}
a &  =0.0938,\\
b &  =-0.5938.
\end{align*}
It can be computed numerically (see Remark \ref{rem:norms}) that
\[
\Vert P_{10}\Vert\approx 0.01530,
\]
up to $5$ decimal places. 
The authors have numerical evidence showing  that
    $$
    D_{{\mathbb{R}},10}(2)\approx \frac{|(0,a,0,b,0,1,0,b,0,a,0)|_{\frac{20}{11}}}{\Vert
    P_{10}\Vert}\approx 90.35556.
    $$
If $\mathbf{a}_{n}$ is the vector of the
coefficients of $P_{10}^{n}$ for each $n\in{\mathbb{N}}$, then we know that
\begin{equation}
D_{{\mathbb{R}},10n}(2)\geq\frac{|\mathbf{a}_{n}|_{\frac{20n}{10n+1}}}{\Vert
P_{10}\Vert^{n}}.\label{equ:case_10}%
\end{equation}
If we set $n=60$ in \eqref{equ:case_10} then we obtain
    $$
    D_{{\mathbb R},600}(2)\geq (1.65171)^{600},
    $$
which suggests that
    $$
    H_{{\mathbb R},\infty}(2)\geq 1.65171.
    $$

We have sketched in Figure \ref{fig:graphs} a summary of the numerical results obtained in this section.

\begin{figure}
\centering
\includegraphics[height=.35\textwidth,keepaspectratio=true]{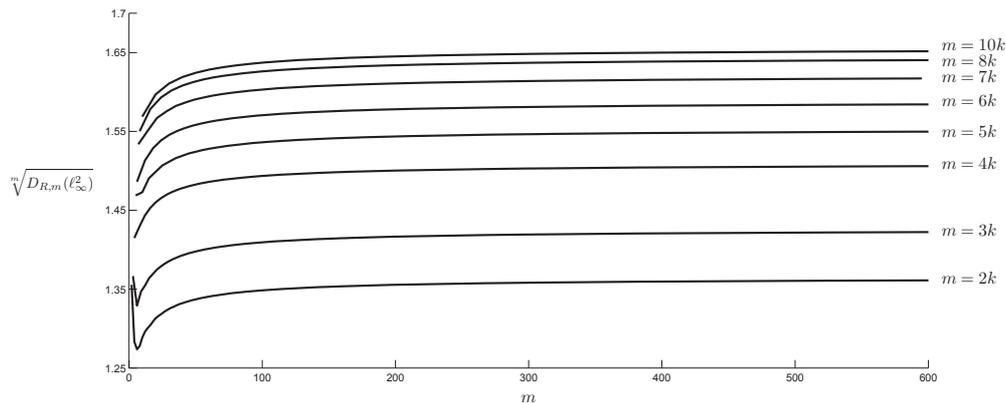}
\caption{Graphs of the estimates on $\sqrt[m]{D_{{\mathbb R},m}(2)}$ obtained by using \eqref{equ:case_2} through \eqref{equ:case_10}.}\label{fig:graphs}
\end{figure}


\begin{bibdiv}
\begin{biblist}
\bib{AronKlimek}{article}{
  author={Aron, R.M.},
  author={Klimek, M.},
  title={Supremum norms for quadratic polynomials},
  journal={Arch. Math. (Basel)},
  volume={76},
  date={2001},
  number={1},
  pages={73--80},
}

\bib{bayart}{article}{
   author={Bayart, F.},
   author={Pellegrino, D.},
   author={Seoane-Sep\'{u}lveda, J.B.},
   title={The Bohr radius of the $n$-dimensional polydisk is equivalent to
   $\sqrt{(\log n)/n}$},
   journal={Adv. Math.},
   volume={264},
   date={2014},
   pages={726--746},
   doi={10.1016/j.aim.2014.07.029},
}

\bib{Blas}{article}{
  author={Blasco, O.},
  title={The Bohr radius of a Banach space},
  conference={title={Vector measures, integration and related topics}, },
  book={ series={Oper. Theory Adv. Appl.}, volume={201}, publisher={Birkh\"{a}user Verlag}, place={Basel}, },
  date={2010},
  pages={59--64},
}

\bib{Boas}{article}{
  author={Boas, H.P.},
  title={The football player and the infinite series},
  journal={Notices Amer. Math. Soc.},
  volume={44},
  date={1997},
  number={11},
  pages={1430--1435},
}

\bib{bh}{article}{
  author={Bohnenblust, H.F.},
  author={Hille, E.},
  title={On the absolute convergence of Dirichlet series},
  journal={Ann. of Math. (2)},
  volume={32},
  date={1931},
  number={3},
  pages={600--622},
}

\bib{bom}{article}{
  author={Bombieri, E.},
  author={Bourgain, J.},
  title={A remark on Bohr's inequality},
  journal={Int. Math. Res. Not.},
  date={2004},
  number={80},
  pages={4307--4330},
}

\bib{LAA2015}{article}{
   author={Campos, J. R.},
   author={Jim\'{e}nez-Rodr{\'{\i}}guez, P.},
   author={Mu\~{n}oz-Fern\'{a}ndez, G.A.},
   author={Pellegrino, D.},
   author={Seoane-Sep\'{u}lveda, J.B.},
   title={On the real polynomial Bohnenblust-Hille inequality},
   journal={Linear Algebra Appl.},
   volume={465},
   date={2015},
   pages={391--400},
   doi={10.1016/j.laa.2014.09.040},
}

\bib{Choi}{article}{
  author={Choi, Y.S.},
  author={Kim, S.G.},
  title={The unit ball of $\scr P(^2\!l^2_2)$},
  journal={Arch. Math. (Basel)},
  volume={71},
  date={1998},
  number={6},
  pages={472--480},
}

\bib{fr}{article}{
  author={Defant, A.},
  author={Frerick, L.},
  title={Hypercontractivity of the Bohnenblust-Hille inequality for polynomials and multidimensional Bohr radii},
  status={arXiv:0903.3395},
}

\bib{annals2011}{article}{
  author={Defant, A.},
  author={Frerick, L.},
  author={Ortega-Cerd\`{a}, J.},
  author={Ouna{\"{\i }}es, M.},
  author={Seip, K.},
  title={The Bohnenblust-Hille inequality for homogeneous polynomials is hypercontractive},
  journal={Ann. of Math. (2)},
  volume={174},
  date={2011},
  number={1},
  pages={485--497},
}

\bib{dgm}{article}{
  author={Defant, A.},
  author={Garc{\'{\i }}a, D.},
  author={Maestre, M.},
  title={Bohr's power series theorem and local Banach space theory},
  journal={J. Reine Angew. Math.},
  volume={557},
  date={2003},
  pages={173--197},
}

\bib{ff}{article}{
  author={Defant, A.},
  author={Garc{\'{\i }}a, D.},
  author={Maestre, M.},
  author={Sevilla-Peris, P.},
  title={Bohr's strips for Dirichlet series in Banach spaces},
  journal={Funct. Approx. Comment. Math.},
  volume={44},
  date={2011},
  number={part 2},
  part={part 2},
  pages={165--189},
}

\bib{dmp}{article}{
  author={Defant, A.},
  author={Maestre, M.},
  author={Prengel, C.},
  title={The arithmetic Bohr radius},
  journal={Q. J. Math.},
  volume={59},
  date={2008},
  number={2},
  pages={189--205},
}

\bib{defant55}{article}{
  author={Defant, A.},
  author={Maestre, M.},
  author={Schwarting, U.},
  title={Bohr radii of vector valued holomorphic functions},
  journal={Adv. Math.},
  volume={231},
  date={2012},
  number={5},
  pages={2837--2857},
}

\bib{defant2}{article}{
  author={Defant, A.},
  author={Sevilla-Peris, P.},
  title={A new multilinear insight on Littlewood's 4/3-inequality},
  journal={J. Funct. Anal.},
  volume={256},
  date={2009},
  number={5},
  pages={1642--1664},
}

\bib{DD}{article}{
  author={Diniz, D.},
  author={Mu\~{n}oz-Fern\'{a}ndez, G. A.},
  author={Pellegrino, D.},
  author={Seoane-Sep\'{u}lveda, J. B.},
  title={The asymptotic growth of the constants in the Bohnenblust-Hille inequality is optimal},
  journal={J. Funct. Anal.},
  volume={263},
  date={2012},
  number={2},
  pages={415--428},
  doi={10.1016/j.jfa.2012.04.014},
}

\bib{diniz2}{article}{
  author={Diniz, D.},
  author={Mu\~{n}oz-Fern\'{a}ndez, G.A.},
  author={Pellegrino, D.},
  author={Seoane-Sep\'{u}lveda, J.B.},
  title={Lower bounds for the constants in the Bohnenblust-Hille inequality: the case of real scalars},
  journal={Proc. Amer. Math. Soc.},
  volume={142},
  date={2014},
  number={2},
  pages={575--580},
  doi={10.1090/S0002-9939-2013-11791-0},
}

\bib{GMS}{article}{
  author={Grecu, B. C.},
  author={Mu\~{n}oz-Fern\'{a}ndez, G. A.},
  author={Seoane-Sep\'{u}lveda, J. B.},
  title={Unconditional constants and polynomial inequalities},
  journal={J. Approx. Theory},
  volume={161},
  date={2009},
  number={2},
  pages={706--722},
}

\bib{Litt}{article}{
  author={Littlewood, J. E.},
  title={On bounded bilinear forms in an infinite number of variables},
  journal={Quart. J. Math. Oxford},
  volume={1},
  date={1930},
  pages={164--174},
}

\bib{monta}{article}{
  author={Montanaro, A.},
  title={Some applications of hypercontractive inequalities in quantum information theory},
  journal={J. Math. Physics},
  volume={53},
  date={2012},
}

\bib{lama11}{article}{
  author={Mu\~{n}oz-Fern\'{a}ndez, G.A.},
  author={Pellegrino, D.},
  author={Seoane-Sep\'{u}lveda, J.B.},
  title={Estimates for the asymptotic behaviour of the constants in the Bohnenblust-Hille inequality},
  journal={Linear Multilinear Algebra},
  volume={60},
  date={2012},
  number={5},
  pages={573--582},
  doi={10.1080/03081087.2011.613833},
}

\bib{Nu3}{article}{
   author={N\'{u}\~{n}ez-Alarc\'{o}n, D.},
   title={A note on the polynomial Bohnenblust-Hille inequality},
   journal={J. Math. Anal. Appl.},
   volume={407},
   date={2013},
   number={1},
   pages={179--181},
   doi={10.1016/j.jmaa.2013.05.008},
}

\bib{addqq}{article}{
  author={Nu\~{n}ez-Alarc\'{o}n, D.},
  author={Pellegrino, D.},
  author={Seoane-Sep\'{u}lveda, J.B.},
  title={On the Bohnenblust-Hille inequality and a variant of Littlewood's 4/3 inequality},
  journal={J. Funct. Anal.},
  volume={264},
  date={2013},
  pages={326--336},
}

\bib{Nun2}{article}{
  author={Nu\~{n}ez-Alarc\'{o}n, D.},
  author={Pellegrino, D.},
  author={Seoane-Sep\'{u}lveda, J.B.},
  author={Serrano-Rodr{\'{\i}}guez, D.M.},
  title={There exist multilinear Bohnenblust-Hille constants $(C_n)_{n=1}^\infty$ with $\lim_{n\rightarrow\infty}(C_{n+1}-C_n)=0$},
  journal={J. Funct. Anal.},
  volume={264},
  date={2013},
  number={2},
  pages={429--463},
  doi={10.1016/j.jfa.2012.11.006},
}

\bib{psseo}{article}{
  author={Pellegrino, D.},
  author={Seoane-Sep\'{u}lveda, J.B.},
  title={New upper bounds for the constants in the Bohnenblust-Hille inequality},
  journal={J. Math. Anal. Appl.},
  volume={386},
  date={2012},
  number={1},
  pages={300--307},
  doi={10.1016/j.jmaa.2011.08.004},
}

\bib{seip3}{article}{
  author={Seip, K.},
  title={Estimates for Dirichlet polynomials},
  journal={EMS Lecturer, CRM},
  date={20--23 February 2012},
  pages={http://www.euro-math-soc.eu/system/files/Seip\_CRM.pdf},
}
\end{biblist}
\end{bibdiv}
\end{document}